\documentclass[12pt,oneside,a4paper,leqno]{article}
\usepackage[utf8]{inputenc}
\usepackage[english]{babel}
\usepackage[all]{xy}

\CompileMatrices

\newdir{ >}{!/10pt/@{ }*@{>}}
\newdir{< }{!/10pt/@{ }*@{<}}

\usepackage{graphicx}
\usepackage{amsfonts,amssymb}
\usepackage[centertags]{amsmath}

\usepackage{amsthm} %
\newcounter{theorems}

\theoremstyle{plain}

\swapnumbers
\newcounter{lemma}
\numberwithin{equation}{section}

\newtheoremstyle{par}%
     {\topsep}%
     {\topsep}%
     {\itshape}%
     {}%
     {\bfseries}%
     {}%
     {.5em}%
     {}%

\newtheoremstyle{parrm}%
     {\topsep}%
     {\topsep}%
     {\normalfont}%
     {}%
     {\itshape}%
     {}%
     {.5em}%
     {}%

\theoremstyle{plain}
\numberwithin{equation}{section}
\newtheorem{lemma}[equation]{Lemma}
\newtheorem{theo}[equation]{Theorem}

\theoremstyle{definition}
\newtheorem{defi}[equation]{Definition}
\newtheorem{example}[equation]{Example}

\theoremstyle{remark}
\newtheorem{remark}[equation]{Remark}

\theoremstyle{par}

\newtheorem{propo}[equation]{}

\theoremstyle{parrm}

\makeatletter
\def\tagform@#1{\maketag@@@{\ignorespaces#1\unskip\@@italiccorr}}

\makeatother

\usepackage{paralist}
\usepackage{comment}
\usepackage{subcaption} %
\newcommand{\RR}{\mathbb{R}}

\newcommand{\from}{\colon}

\usepackage{bm}
\newcommand{\simbolovettore}[1]{{\boldsymbol{#1}}}

\newcommand{\vc}{\simbolovettore{c}}

\newcommand{\vm}{\simbolovettore{m}}
\newcommand{\vn}{\simbolovettore{n}}

\newcommand{\vq}{\simbolovettore{q}}

\newcommand{\vx}{\simbolovettore{x}}

\newcommand{\vQ}{\simbolovettore{Q}}

\newcommand{\vL}{\simbolovettore{L}}
\newcommand{\vY}{\simbolovettore{Y}}

\newcommand{\zero}{\boldsymbol{0}}
\newcommand{\abs}[1]{\lvert{#1}\rvert}

\newcommand{\tT}{\mkern-1.5mu\mathsf{T}}
\newcommand{\pfaff}{\operatorname{Pf}}

\newcommand{\bordered}[1]{{#1}^{\mathsf{b}}}

\usepackage{multirow,rotating}

\usepackage{listings}
\newcommand{\term}[1]{\emph{#1}}
\newcommand{\norm}[1]{\lVert{#1}\rVert}

\newcommand{\conf}[2]{\mathbb{F}_{#1}(#2)}

\newcommand{\CH}{\operatorname{CH}}

\begin{document}
\pagenumbering{arabic}

\title{%
Pfaffians and the inverse
problem for collinear central configurations 
}

\author{D.L.~Ferrario}

\date{%
\today}
\maketitle

\begin{abstract}
We consider, after Albouy--Moeckel, the inverse problem for col\-li\-ne\-ar central
configurations: given a collinear configuration of $n$ bodies, find  
positive
masses which make it central.  We 
give some new estimates concerning the
positivity of Albouy-Moeckel pfaffians: we show that for any homogeneity
$\alpha$ and $n\leq 6$ or $n\leq 10$ and $\alpha=1$ (computer-assisted) the
pfaffians are positive.  Moreover,
for the inverse problem with  positive masses, we show that for any
homogeneity and $n\geq 4$ there are explicit regions of the configuration space
without solutions of the inverse problem.

\end{abstract}

\section{Introduction}

Let $n\geq 2$, and $d\geq 1$. The \emph{configuration} space of $n$ points in 
the $d$-dimensional euclidean space $E=\RR^d$ is defined as
\[
\conf{n}{E} = \{ \vq \in E^n : \vq_i \neq \vq_j \},
\]
where $\vq = (\vq_1, \vq_2, \ldots, \vq_n)\in E$ and $\forall j, \vq_j\in E$.
Given a positive parameter $\alpha>0$, and $n$ positive masses $m_j>0$, 
the \emph{potential function} $U\from \conf{n}{E} \to \RR$ is defined 
as 
\[
U(\vq) = \sum_{1\leq i < j \leq n} \dfrac{m_im_j}{\norm{\vq_i - \vq_j}^\alpha}.
\]
A \emph{central configuration} is a configuration 
that yields a relative equilibrium solution 
of the Newton equations of the $n$-body problem with potential function $U$, 
and can be shown  (cf. 
\cite{Moultonstraightlinesolutions1910},
\cite{moeckelCentralConfigurations1990},
\cite{AlbouyInverseProblemCollinear2000},
\cite{FerrarioFixedpointindices2015},
\cite{ferrarioCentralConfigurationsMorse2017},
\cite{ferrarioCentralConfigurationsMutual2017})
that it is a solution of the following $n$ equations
\begin{equation}
\lambda m_j \vq_j = - \alpha \sum_{k\neq j} m_j m_k \dfrac{\vq_j - \vq_k}{\norm{\vq_j - \vq_k}^{\alpha+2}}.
\end{equation}
Such configurations have center of mass $\sum_{j=1}^n m_j \vq_j = \zero\in E$, 
and the parameter $\lambda$ turns out to be equal to 
\(
\lambda = -\alpha \dfrac{U(\vq)}{\sum_{j=1}^n m_j \norm{\vq_j}^2 }.
\)
A generic central configuration  (with center of mass 
\(
\vq_0 = \dfrac{\sum_{j=1}^n m_j \vq_j}{M}
\)
not necessarily $\zero$,
where $M=\sum_{j=1}^n m_j$)
satisfies the equation 
\begin{equation}
\label{eq:CC}
\lambda m_j (\vq_j-\vq_0) = - \alpha \sum_{k\neq j} m_j m_k \dfrac{\vq_j - \vq_k}{\norm{\vq_j - \vq_k}^{\alpha+2}}.
\end{equation}

Now, if for each $i,j$ denote  
\[
\vQ_{jk} = \dfrac{\vq_j - \vq_k}{\norm{\vq_j - \vq_k}^{\alpha+2}},
\]
equation \eqref{eq:CC}
can be written as
\footnote{%
In the notation of \cite{AlbouyInverseProblemCollinear2000}, 
$\vq_j = X_j$, $\vq_0=c$,   
$A_j=\sum_{k\neq j} m_k \vQ_{ki}$, 
so that the equation \eqref{eq:CC} reads as equation (3) of  \cite{AlbouyInverseProblemCollinear2000}
$\alpha A_j - \lambda(\vq_j-\vq_0) =\zero$, $j=1,\ldots, n$, 
for some constant $\lambda<0$. 
}
\begin{equation}
\label{eq:CC2}
\vq_j = M^{-1} \sum_{k=1}^n m_k \vq_k - \dfrac{\alpha}{\lambda} \sum_{k\neq j} m_k \vQ_{jk} ,
\quad j=1,\ldots, n.
\end{equation}

The \emph{inverse problem}, introduced by Moulton
\cite{Moultonstraightlinesolutions1910} (see also
Buchanan
\cite{Buchanancertaindeterminantsconnected1909}), and considered by Albouy and Moeckel
in \cite{AlbouyInverseProblemCollinear2000}, 
can be phrased as follows: given the positions $\vq_j$
(or, equivalently, the mutual differences $\vq_i-\vq_j$)
to find the (positive) masses $m_j$ and $\lambda<0$ such that \eqref{eq:CC2} holds.
As it is, the equation is not linear in the $(n+1)$-tuple  $(m_1,\ldots, m_n,\lambda)$, 
but can be transformed into the following equation
\begin{equation}
\label{eq:CC3}
\vq_j =  \hat\vc + \sum_{k\neq j} \hat m_k \vQ_{jk} ,
\quad j=1,\ldots, n,
\end{equation}
because of the following lemma. 

\begin{lemma}
\label{eq:CC3-CC2}
Given $\vq\in \conf{n}{E}$, there exists $(m_1,\ldots, m_n,\lambda)$,
with $m_j>0$ satisfying
$\eqref{eq:CC2}$ if and only if there exists
$(\hat m_1,\ldots, \hat m_n, \hat \vc) \in \RR^{n+d}$ such that
\eqref{eq:CC3} holds 
and $\hat m_j >0 $ for each $j$. 
\end{lemma}
\begin{proof}
If \eqref{eq:CC2} holds for $(m_1,\ldots, m_n,\lambda)$ with positive masses, 
then $\lambda<0$ and simply by setting
\(
\hat \vc  = 
M^{-1} \sum_{k=1}^n m_k \vq_k ~, \quad  
\hat m_k  = - \dfrac{\alpha}{\lambda}  m_k 
\)
one has that \eqref{eq:CC3} holds. 

Conversely, assume that $(\hat m_1,\ldots, \hat m_n,\hat \vc)$ satisfies
\eqref{eq:CC3}, with $\hat m_j>0$. Then by putting
\(
m_k  = \hat m_k, \quad k=1,\ldots, n ~,\quad
\lambda  = -\alpha %
\)
it follows, multiplying by $m_j$ (and setting as above $M=\sum_{j=1}^n m_j$) 
and summing for $j=1,\ldots, n$
\[
\vq_j  =  \hat\vc -\dfrac{\alpha}{\lambda} \sum_{k\neq j} m_k \vQ_{jk}, \quad
\implies \quad
\sum_{j=1}^n m_j \vq_j  = M \hat \vc + \zero, 
\]
and hence 
\eqref{eq:CC2}.
\end{proof}

\begin{remark}
Multiplying each equation by $\hat m_j (\vq_j - \hat\vc)$, and summing for $j=1,\ldots, n$, it follows that
\(
\sum_{j=1}^n \hat m_j \norm{\vq_j-\hat \vc}^2 %
= \sum_{j=1}^n \sum_{k\neq j} \hat m_j \hat m_k \vQ_{jk} \cdot (\vq_j - \hat \vc) %
 = 
\sum_{1\leq j < k \leq n}\hat m_j \hat m_k \norm{\vq_j -\vq_k}^{-\alpha}. %
\)
Hence whenever \eqref{eq:CC2} or \eqref{eq:CC3} holds (for positive masses), 
the corresponding $\lambda$ is in any case negative.
Moreover, \eqref{eq:CC2} holds for $(m_1,\ldots, m_n,\lambda)$ if and only if 
it holds for $(t m_1,\ldots, t m_n,t\lambda)$ for any $t>0$, so that equations \eqref{eq:CC2}
and \eqref{eq:CC3} are %
equivalent.  
\end{remark}

\begin{defi}
For each $\vq\in \conf{n}{E}$, let 
$\Psi(\vq), \tilde\Psi(\vq) \subset E^n$ be the subsets
\[
\begin{aligned}
\Psi(\vq) & = \{ \vq : \vq_j =   
\hat\vc + \sum_{k\neq j} \hat m_k \vQ_{jk} : 
\hat \vc \in E, \hat m_j >0, j=1,\ldots, n
\} \\
\subseteq
\tilde \Psi(\vq) &  = \{   \vq : \vq_j = 
\hat\vc + \sum_{k\neq j} \hat m_k \vQ_{jk} : 
\hat \vc \in E, \hat m_j \in \RR, j=1,\ldots, n
\}.
\end{aligned}
\] 
Hence, given $\vq\in \conf{n}{E}$, there exists a solution of \eqref{eq:CC3} if and only 
if $\vq \in \Psi(\vq)$; furthermore, if $\vq\in \Psi(\vq)$ then
  $\vq\in \tilde\Psi(\vq)$. 
\end{defi}

We will now deal with the collinear case. First, we will follow Albouy--Moeckel 
\cite{AlbouyInverseProblemCollinear2000}
and consider the inverse problem with real masses; 
then we will consider the problem with positive masses, and 
follow Ouyang--Xie \cite{OuyangCollinearCentralConfiguration2005} (for $n=4$ 
bodies and $\alpha=1$) and Davis et al. \cite{DavisInverseproblemcentral2018a} (for $n=5$
bodies and $\alpha=1$),
in understanding in which regions the inverse problem has no solutions.

\section{The case $d=1$: collinear configurations and Pfaffians}

For $d=1$, all configurations are on a line, therefore
 $E=\RR$, $c=\vc$, and $\vq \in \tilde\Psi(\vq)$ if and only if 
there exists 
$(m_1,...,m_n,c) \in \RR^{n+1}$ such that
\begin{equation}
\label{eq:d1}
 \begin{bmatrix}
 0 & Q_{12} & Q_{13} & \ldots & Q_{1n} \\
 -Q_{12} & 0 & Q_{23} & \ldots & Q_{2n} \\
 \vdots & \vdots & \vdots &\ddots & \vdots \\
 -Q_{1n} & -Q_{2n} & \ldots & -Q_{n-1,n} & 0
 \end{bmatrix}
 \begin{bmatrix}
 m_1 \\ m_2 \\ \vdots \\ m_n
 \end{bmatrix}
 +
 c 
 \begin{bmatrix}
 1\\1\\\vdots\\1
 \end{bmatrix}
 =
 \begin{bmatrix}
 q_1\\q_2\\\vdots\\q_n
 \end{bmatrix}
\end{equation}
where $Q_{ij} =(q_i - q_j)\abs{q_i - q_j}^{-\alpha-2}$, 
for $i,j=1,\ldots, n$, or  
\begin{equation}\label{eq:main1d}
\iff 
Q \vm  + c \vL = \vq
\end{equation}
where $Q$ is the $n\times n$ skew-symmetric matrix with entries $Q_{ij}$,
$\vm$ the vector of masses, and $\vL$ the vector with constant
components $1$.

Recall %
that if $n$ is odd and $A$ is an anti-symmetric  
$n\times n$ matrix $A^{\tT} = - A \implies \det(A) = \det(-A) = (-1)^n \det(A) \implies \det(A) = 0$.
If $n$ is even, then
$\det A  = \left( 
\pfaff A
\right)^2$
(cf. for example the combinatorial approach of \cite{godsilAlgebraicCombinatorics1993}, Chap. 7, 
or the multi-linear algebra approach of \cite{northcottMultilinearAlgebra1984}, from page 100).
The \term{pfaffian}  $\pfaff Q$  of a skew-symmetric matrix $Q$ (for even $n$) is defined as follows
(in Moulton's 1910 notation): 
\setlength{\arrayrulewidth}{1pt}
\[
\pfaff Q = 
\begin{array}{cccc|}
\multicolumn{1}{|c}{Q_{12}} & Q_{13} & \cdots &  Q_{1n}\\
 & Q_{23} & \cdots & Q_{2n} \\
 & & \ddots & \vdots \\
 && & Q_{n-1,n}
\end{array}
=
\sum_{\sigma} (-1)^\sigma  Q_{r_1,s_1} Q_{r_2,s_2} \ldots Q_{r_k,s_k}
\]
where $n=2k$, and the permutation  
$\sigma$ runs over all \emph{perfect matchings} of $\vn=\{1,\ldots, n=2k\}$:
a perfect matching $\sigma$ is a fixpoint free involution of $\vn$,
which can be represented also as 
a partition of $\vn$ in pairs $[r_1,s_1, r_2,s_2,\ldots, r_k,s_k]$. 
The sign $(-1)^\sigma$ is the parity of this permutation.
In D. Knuth and Cayley notation 
\cite{knuthOverlappingPfaffians1996,cayley1849determinants}
\( \pfaff A = A[1,2,\ldots, n]  \).

The following recursive identity is the analogue of the Laplace expansion for the determinant: 
\begin{equation}\label{eq:pfaff1}
A[1,2,\ldots,n]=\sum_{j=1}^{n-1} (-1)^{j+1} A_{jn} 
 A[1,\ldots, \hat j, \ldots, \hat n],
\end{equation}
where $A[1,\ldots, \hat j, \ldots, \hat n]$ denotes the Pfaffian of the matrix
 with the $j$-th and $n$-th rows and columns canceled out.

An elementary property of Pfaffians is the following: if $A$ is a skew-symmetric matrix,
and $B$ the matrix obtained by swapping the $i$-th and $j$-th columns and the $i$-th 
and $j$-th rows, then 
\begin{equation}
\label{eq:AminusB}
\pfaff A = - \pfaff B. 
\end{equation}

\begin{lemma}[Halton]
\label{lemma:halton}
Let $A$ be an $n\times n$ skew-symmetric matrix, and $i<j$, with $n$ even.  
If $A_{ij}$ denotes the matrix $A$ with row $i$ and column $j$ removed,  then
\begin{equation}\label{eq:halton}
\det A_{ij} = - 
 A[1,\ldots, \hat i, \ldots, \hat j, \ldots, n] \, \pfaff A.
\end{equation}
\end{lemma}

\begin{remark}
See for example lemma 3.2 at page 118 of \cite{godsilAlgebraicCombinatorics1993}, for a proof,
where it is used to prove the recursive relation of Pfaffians. 
See also 
\cite{stembridgeNonintersectingPathsPfaffians1990a},
\cite{DressSimpleProofIdentity1995}, 
\cite{HamelPfaffianIdentitiesCombinatorial2001} for other interesting combinatorial identities 
for pfaffians. 
\end{remark}

\begin{remark}[Buchanan Albouy--Moeckel Conjecture]
Buchanan%
, in his 1909 article \cite{Buchanancertaindeterminantsconnected1909},
proves a proposition which can be rephrased as follows:
\emph{for each even $n$, $\alpha=1$, 
for each $\vq\in \conf{n}{\RR}$, the Pfaffian is non-zero: $\pfaff A_n \neq 0$.}

As found by Albouy and Moeckel in \cite{AlbouyInverseProblemCollinear2000}, 
Buchanan's proof uses an incorrect argument, and cannot be repaired. So, they
conjecture it to be true, in the 
\emph{Albouy--Moeckel Conjecture}: the Pfaffians are non-zero for all %
configurations. 
The partial steps done in the direction of its complete proof are the following:
it is true for $n\leq 4$ and $\alpha>0$, 
or $\alpha=1$ and $n\leq 6$,  computer-assisted  
(Albouy-Moeckel 2000  \cite{AlbouyInverseProblemCollinear2000}); 
it is true for $n\leq 6$ and $\alpha=1$ (Xie 2014 \cite{Xieanalyticalproofcertain2014}).
\end{remark}

The following lemma generalizes Theorem 2.4.(1-2) of 
\cite{Xieanalyticalproofcertain2014};
the main conclusion
follows from Proposition 5 of \cite{AlbouyInverseProblemCollinear2000}.

\begin{lemma}
\label{lemma:Pf>0}
If $q_1>q_2>q_3>q_4$ and as above 
$Q_{ij} =q_{ij}\abs{q_{ij}}^{-\alpha-2}$, 
then
$Q_{12}Q_{34}>Q_{13}Q_{24}$, 
and $ Q_{23}Q_{14} > Q_{13} Q_{24}$, and hence
\[
Q_{12} Q_{34} - Q_{13}Q_{24} + Q_{23}Q_{14} > 0.
\]
\end{lemma}
\begin{proof}
\[
\begin{aligned}
Q_{23} Q_{14} > Q_{13} Q_{24}  \iff  & 
(q_{23} q_{14} )^{-\alpha-1}  > (q_{13} q_{24})^{-\alpha-1} 
\\
    \iff & 
q_{23} q_{14}   < q_{13} q_{24} \\
 \iff  & 
q_{23} (q_{13} + q_{34}) < q_{13}  (q_{23} + q_{34}) 
\\ 
  \iff &
\dfrac{ q_{13} + q_{34} }{ q_{13}}  < \dfrac{q_{23} + q_{34}}{ q_{23} } \\
 \iff  &
1 + \dfrac{q_{34} }{ q_{13}}  <  1+ \dfrac{q_{34}}{ q_{23} }~, \\
\end{aligned}
\]
and the last inequality holds true since $q_{13} > q_{23}$. 
Now, this implies 
\[
Q_{12}Q_{34} - Q_{13} Q_{24}  + Q_{23} Q_{14} > Q_{12}Q_{34} > 0 ~. 
\]
\end{proof}

The following lemma generalizes Theorem 2.4.(3) of 
\cite{Xieanalyticalproofcertain2014}.

\begin{lemma}
\label{lemma:Pf'>0}
Assume $q_1>q_2>q_3>q_4$, 
and as above 
$Q_{ij} =q_{ij}\abs{q_{ij}}^{-\alpha-2}$.
The function $f(q_4) = \pfaff A_4 = Q_{14}Q_{23} -Q_{24}Q_{13} +Q_{34} Q_{12}$
is monotone increasing in $(-\infty,q_3)$, with $q_1,q_2,q_3$ fixed. 
The function $g(q_1) = \pfaff A_4$ is monotone decreasing in $(q_2,+\infty)$,
with $q_2,q_3,q_4$ fixed. 
\end{lemma}
\begin{proof}
\[
\begin{aligned}
\dfrac{d (\pfaff A_4) }{d q_4} & =  (\alpha+1) \left( q_{14}^{-\alpha-2} Q_{23} - q_{24}^{-\alpha-2} Q_{13}
+ q_{34}^{-\alpha-2} Q_{12} \right) \\
& = (\alpha+1) \left(
 Q_{12}\dfrac{Q_{34}}{q_{34}} - Q_{13} \dfrac{Q_{24}}{q_{24}}  + Q_{23} \dfrac{Q_{14}}{q_{14}} 
\right) 
\\
\end{aligned}
\]
Since $Q_{12}Q_{34}>Q_{13}Q_{24}$ 
and $ Q_{23}Q_{14} > Q_{13} Q_{24}$ by \ref{lemma:Pf>0}, 
\[
\begin{aligned}
 Q_{12}\dfrac{Q_{34}}{q_{34}} - Q_{13} \dfrac{Q_{24}}{q_{24}}  + Q_{23} \dfrac{Q_{14}}{q_{14}} & > 
 \dfrac{1}{q_{34}} Q_{13}Q_{24} - \dfrac{1}{q_{24}} Q_{13} Q_{24}  + \dfrac{1}{q_{14}} Q_{23}Q_{14} \\
 & = \left( \dfrac{1}{q_{34}} - \dfrac{1}{q_{24}} + \dfrac{1}{q_{14}} \right) Q_{13} Q_{24} >0.
\end{aligned}
\]
The second part of the statement follows by considering that if $q_1>q_2>q_3>q_4$,
then one can define $y_1=-q_4>y_2=-q_3>y_3=-q_2>y_4=-q_1$, and the Pfaffian of the 
corresponding matrix $Y_{ij} = (y_{ij})\abs{y_{ij}}^{-\alpha-2}$,
with $y_{ij} = y_i - y_j$,  is equal to
\[
\begin{array}{ccc|}
\multicolumn{1}{|c}{Y_{12}} &  Y_{13} &  Y_{14}\\
 &  Y_{23} &  Y_{24} \\
 &&  Y_{34}
\end{array}  
= 
\begin{array}{ccc|}
\multicolumn{1}{|c}{Q_{34}} &  Q_{24} &  Q_{14}\\
 &  Q_{23} &  Q_{13} \\
 &&  Q_{12}
\end{array}  
= 
\begin{array}{ccc|}
\multicolumn{1}{|c}{Q_{12}} &  Q_{13} &  Q_{14}\\
 &  Q_{23} &  Q_{24} \\
 &&  Q_{34}.
\end{array}  
\]
Since $f(y_4)$ is monotonically increasing in $(-\infty,y_3)$, and $y_3=-q_2$, the function $g(q_1) = f(-y_4)$  
is monotonically decreasing in $(q_2,+\infty)$. 
\end{proof}

The following lemma is inspired by the proof of Theorem 2.5 of  
\cite{Xieanalyticalproofcertain2014},
and in fact generalizes it.

\begin{lemma}
\label{lemma:crisscross}
If $\vq\in \RR^n$ is a (collinear) configuration with $q_1>q_2>\ldots>q_n$, and $Q$ 
denotes the skew-symmetric matrix with entries $Q_{ij}$, then 
\[
\pfaff Q = 
\begin{array}{cccc|}
\multicolumn{1}{|c}{Q_{12}} & Q_{13} & \cdots &  Q_{1n}\\
 & Q_{23} & \cdots & Q_{2n} \\
 & & \ddots & \vdots \\
 && & Q_{n-1,n}
\end{array}
=
\left(
\prod_{j=1}^{n-1} Q_{jn}
\right) \cdot \left( 
\begin{array}{cccc|}
\multicolumn{1}{|c}{\tilde Q_{12}} & \tilde Q_{13} & \cdots &  1 \\
 & \tilde Q_{23} & \cdots & 1  \\
 & & \ddots & \vdots \\
 && & 1
\end{array} \right) ~,
\]
where for each $i,j = 1,\ldots, n-1$ 
\[
\tilde Q_{ij} = 
 \left( q^{-1}_{jn}-q^{-1}_{in} \right)^{-\alpha-1}.
\]
Hence, 
if the  configuration $\tilde \vq \in \RR^{n-1}$ is 
defined by  $\tilde q_j =  - q^{-1}_{jn}$ for each $j=1,\ldots, n-1$, it satisfies
\[
\tilde q_1 >
\tilde q_2 >
\ldots >
\tilde q_{n-1}
\]
and,  as for $Q$, with $\tilde q_{ij} = \tilde q_i - \tilde q_j$,
$\tilde Q_{ij} =\tilde q_{ij}\abs{\tilde q_{ij}}^{-\alpha-2}$.
\end{lemma}
\begin{proof}
By multiplying on the left and the right the matrix $Q$ with the $n\times n$
matrix with diagonal $(Q_{1n}^{-1},Q_{2n}^{-1},\ldots, Q_{n-1,n}^{-1},1)$, 
one obtains a matrix $\tilde Q$ with entries
\[
\tilde Q_{ij} = \begin{cases}
\frac{Q_{ij}}{Q_{in}Q_{jn}} & \text{ if } 1\leq i,j \leq n-1 \\
\end{cases}
\]
and the proof follows 
from the fact that if $i<j$ then 
\[
\frac{Q_{ij}}{Q_{in}Q_{jn}}  
= \left( \frac{q_{ij}}{q_{in}q_{jn}}\right)^{-\alpha-1}  =
 \left( \frac{q_{in}-q_{jn} }{q_{in}q_{jn}}\right)^{-\alpha-1}  =
 \left( q^{-1}_{jn}-q^{-1}_{in} \right)^{-\alpha-1}.
\]
\end{proof}

Given an $n\times n$ skew-symmetric matrix $Q$,
let $\bordered{Q}$ denote the $(n+1)\times(n+1)$ skew-symmetric bordered matrix 
\[
\bordered{Q} =
\begin{bmatrix}
0 & Q_{12} & Q_{13} & \ldots & Q_{1n} & 1 \\
-Q_{12} & 0 & Q_{23} & \ldots & Q_{2n}  & 1\\
\vdots & \vdots & \vdots &\ddots & \vdots & \vdots\\
-Q_{1n} & -Q_{2n} & \ldots & -Q_{n-1,n} & 0 & 1 \\
-1  & -1  & \ldots & -1  & -1 & 0 \\
\end{bmatrix}~.
\]
With this notation, lemma \ref{lemma:crisscross} can be written as 
\(
\pfaff Q = \left( 
\prod_{j=1}^{n-1} Q_{jn}
\right) \pfaff \bordered{\tilde Q}  \).

\begin{propo}
\label{propo:odd}
If $n$ is odd, and for $\vq\in \conf{n}{\RR}$ the product of pfaffians
\[
 \bordered{Q}[1,\ldots, n,n+1]   Q[1,\ldots, n-1,\hat n] \neq 0
\]
is non-zero, then
equation \eqref{eq:main1d} has solutions. 
\end{propo}
\begin{proof}
Observe that equation \eqref{eq:main1d} has solutions if 
the rank of the $n\times (n+1)$ matrix 
\[
\begin{bmatrix}
0 & Q_{12} & Q_{13} & \ldots & Q_{1n} & 1 \\
-Q_{12} & 0 & Q_{23} & \ldots & Q_{2n}  & 1\\
\vdots & \vdots & \vdots &\ddots & \vdots & \vdots\\
-Q_{1n} & -Q_{2n} & \ldots & -Q_{n-1,n} & 0 & 1 \\
\end{bmatrix}
\]
is equal to $n$, which happens if for some $j \in \{1,\ldots, n\}$ the $n\times n$ square matrix
obtained by removing the $j$-th column is non-singular. 
Now, this is the same as the matrix obtained by 
removing the $(n+1)$-th row and the $j$-th column of 
the bordered matrix $\bordered{Q}$. By \eqref{eq:halton} (on transposed matrices)
its determinant is equal to
\[
Q[1,\ldots, \hat j, \ldots,  n] \pfaff \bordered{Q}.
\]
By taking $j=n$ the conclusion follows. 
\end{proof}

Note that that statement holds 
with  $j$ chosen  as any index from $1$ to $n$, instead of $n$;
moreover, because of \eqref{eq:pfaff1}, there exists  $j$ such that 
$ \bordered{Q}[1,\ldots, n,n+1]   Q[1,\ldots, \hat j,  \ldots, n] \neq 0$
if and only if 
$ \bordered{Q}[1,\ldots, n,n+1]   \neq 0$. 
See also Theorem 1 of \cite{AlbouyInverseProblemCollinear2000}, %
where 
shorter proofs or more general results are presented,
using exterior algebra as a computational device.

Let $n$ be odd and $\vq$ a configuration. Then the corresponding 
$Q_n$ is a $n\times n$ singular matrix. The two matrices in 
\ref{propo:odd} are the $(n-1)\times (n-1)$ skew-symmetric matrix $Q_{n-1}$ correponding to 
the configuration with the $n$-th body removed, and the $(n+1)\times(n+1)$ matrix  
$\bordered{Q_n}$. Because of \ref{lemma:crisscross}, 
the pfaffian $\pfaff Q_{n-1}$ is non-zero if and only if the pfaffian 
of the corresponding $\bordered{\tilde Q_{n-1}}$ is non-zero. But 
$\tilde Q_{n-1}$ is 
an $(n-2)\times(n-2)$ matrix.
So, for odd $n$ the existence of solutions to \eqref{eq:main1d} follows 
from the calculation of pfaffians of the even-dimensional matrices $\bordered{Q_{n}}$ 
and $\bordered{\tilde Q_{n-1}}$ 
(the existence of solutions for $n=5$ was proven in Theorem 2.6 of 
\cite{Xieanalyticalproofcertain2014} in a different way).  

On the other hand, let $n$ be even, and $\vq$ a configuration and $Q_n$ as above. 
By \ref{lemma:crisscross} the existence of solutions to \eqref{eq:main1d} follows 
from the calculation of the pfaffian of $\bordered{\tilde Q_{n}}$, 
where $\tilde Q_n$ is a matrix with odd size.

\begin{theo}
For all $\alpha>0$, and any $n\leq 6$, 
the pfaffian of $Q$ (for even $n$) or of $\bordered{Q}$ (for odd n)
is non-zero, 
hence for each configuration $\vq$ equation \eqref{eq:main1d} has solutions  
with real masses $m_j$. 
\end{theo}
\begin{proof}
By lemma \ref{lemma:crisscross}, as explained before, the pfaffian of the matrix 
corresponding to a collinear configuration $\vq\in \conf{n}{\RR}$ 
with $n$ even is non-zero, if it is non-zero the pfaffian of the bordered 
matrix $\bordered{Q}$ corresponding to collinear $n-1$ bodies. 
For $n=5$ one can apply \eqref{eq:pfaff1} and obtain, 
given that $\bordered{Q}_{j6} = 1$ for $j=1,\ldots, 5$, 
\[
\begin{aligned}
\pfaff \bordered{Q} & = 
 Q[\hat 1, 2,3,4,5] 
-  Q[1,\hat 2, 3,4,5]  \\
& +  Q[1,2,\hat 3, 4,5] 
-  Q[1,2,3,\hat 4, 5]  
+  Q[1,2,3,4,\hat 5].
\end{aligned}
\]
Without loss of generality one can assume $q_1>q_2\ldots > q_5$:
since by lemma \ref{lemma:Pf'>0} the pfaffian $Q[2,3,4,5]$ is decreasing in $q_2$, and $q_1>q_2$, 
one has $Q[1,3,4,5]  < Q[2,3,4,5]$;
since $Q[1,2,3,4]$ is increasing in $q_4$, and $q_4>q_5$, 
$Q[1,2,3,4] > Q[1,2,3,5]$. Therefore $\pfaff \bordered{Q} > Q[1,2,\hat 3, 4,5]$,
which is strictly positive by \ref{lemma:Pf>0}.  
\end{proof}

\begin{remark}
Such a nice argument, introduced already by Xie in 
\cite{Xieanalyticalproofcertain2014}, unfortunately does not work as it is for $n>6$: 
when $n\geq 8$ in the (symmetric) sum of 7 terms only the two consecutive terms at both 
endpoints
can be estimated by monotonicity. It is very interesting that, at least for $\alpha=1$ 
when the pfaffian is a rational function of the mutual distances, 
it is possible to prove its positivity by checking that \emph{all the coefficients} 
of the polynomials are positive. This was found by Albouy and Moeckel in 
\cite{AlbouyInverseProblemCollinear2000}: in the following we show how we computed 
the polynomial for $n=8$ and $10$, finding that it has all positive coefficients.  

It is maybe worth noting that in the notation of 
\cite{AlbouyInverseProblemCollinear2000} the following equalities 
hold: if $n=2k$ then 
$ K_n = k! \pfaff Q$ 
while if $n=2k+1$,  then
$K_n^L = k! \pfaff \bordered{Q}$. 
\end{remark}

\begin{lemma}
\label{lemma:polynomial}
Let $\alpha=1$, $\vq\in \conf{n}{\RR}$ an ordered collinear configuration (with 
$q_1>q_2>\ldots>q_n$, and as above $q_{ij}=q_i-q_j$), and $n$ even.  
Let $P$ be  the skew-symmetric matrix 
defined for each $i<j$ by $P_{ij} = $ the product of all $q_{ab}$ such that $a\in \{i,j\}$ or $b\in \{i,j\}$
and $a<b$:
\[
P_{ij} = \prod_{\substack{1\leq a<b\leq n\\ \{a,b\}\cap \{i,j\}\neq\emptyset \\ (a,b) \neq (i,j)}} q_{ab}.
\]
Its pfaffian and the pfaffian of 
the anti-symmetric matrix with terms $Q_{ij} = q_{ij}^{-2}$ for $i<j$
satisfy the identity
\[
\pfaff P = \left( \prod_{1\leq i<j\leq n} q_{ij}^2 \right) \pfaff Q.
\]
\end{lemma}
\begin{proof}
Let $P'$ denotes the matrix obtained by 
multiplying the $j$-th row and column of $Q$ by the factor
$\displaystyle(-1)^{j-1}\prod_{\substack{1\leq i \leq n\\ i\neq j}} q_{ij}$, for $j=1,\ldots, n$. 
It follows that 
\[\begin{aligned}
\pfaff P' & = 
\left( \prod_{1\leq j \leq n} (-1)^{j-1} \prod_{\substack{1\leq i \leq n\\ i\neq j}} q_{ij}\right)
\pfaff Q %
= %
\left( 
\prod_{1\leq i<j\leq n} q_{ij}^2 
\right)
\pfaff Q\\
\end{aligned}\]
since
\[
\prod_{\substack{1\leq i \leq n\\ i\neq j}} q_{ij} = 
\left( \prod_{1\leq i < j} q_{ij} \right) 
\left( \prod_{j<i\leq n} q_{ij} \right)
= 
(-1)^{n-j-1}
\left( \prod_{1\leq i < j} q_{ij} \right) 
\left( \prod_{j<i\leq n} q_{ji} \right).
\]
This implies also that 
the $ij$-entry of $P'$ is equal to 
\[
\begin{aligned}
P'_{ij} & = q_{ij}^{-2} 
\left( \prod_{\substack{1\leq a<b\leq n\\ i\in \{a,b\} } } q_{ab} \right)
\left( \prod_{\substack{1\leq a<b\leq n\\ j\in \{a,b\} } } q_{ab} \right) %
= P_{ij}~.
\end{aligned}
\qedhere
\]
\end{proof}

\begin{remark}
\label{remark:polynomial}
For even $n$, if the matrix $P$ of \ref{lemma:polynomial} is computed starting 
from
the matrix $\bordered{\tilde Q}$ 
of \ref{lemma:crisscross} instead of $Q$, it can be renamed  $\tilde P$: 
its pfaffian is a polynomial in the $n-2$ variables
$\tilde {x}_j=\tilde q_{j}-\tilde q_{j+1}=\tilde q_{j,j+1}$ for $j=1,\ldots, n-2$,
where $\tilde q_j = -q_{jn}^{-1}$, 
and for each $1\leq i<j < n$ the equality $\tilde q_{ij} = \frac{q_{ij}}{q_{in}{q_{jn}}}$ holds,
and $\tilde q_{in} = 1$.  
Note that $\tilde q_n$ is not defined, 
and $\tilde q_{in}$ is \emph{not} $\tilde q_{i} - \tilde q_{n}$;
hence  
 $\tilde q_{in} = 1$, 
for $i=1,\ldots, n-1$, does not imply $\tilde q_1 = \ldots = \tilde q_{n-1}$. 
\end{remark}

\begin{theo}
\label{theo:code}
The pfaffian of the matrix $P$, defined in \ref{lemma:polynomial}, is a polynomial 
with non-negative integer coefficients, for each even $n\leq 8$, 
with respect to the variables $x_1,\ldots, x_{n-1}$, 
defined as $x_j=q_{j}-q_{j+1}=q_{j,j+1}$ for $j=1,\ldots, n-1$. 

The pfaffian of the matrix $\tilde P$, defined in \ref{remark:polynomial}, 
is a polynomial with non-negative integer coefficients, for each even $n\leq 10$,
with respect to the variables $\tilde x_1,\ldots, \tilde x_{n-2}$, 
defined as $\tilde x_j=\tilde q_{j}-\tilde q_{j+1}=\tilde q_{j,j+1}$ for $j=1,\ldots, n-2$,
where $\tilde  q_j = - q_{jn}^{-1}$.  

As a consequence, for each even $n\leq 10$ the pfaffian of $Q$ is positive. 
\end{theo}
\begin{proof}[Proof (computer assisted)]

The proof is just a computer computation, performed on some computer algebra systems. 
The output numbers for the first cases are as follows.

For $P$:

$n=4$: minumum of coefficients $=1$, maximum of coefficients = 19. 
Total of 25 non-zero coefficients in the $n-1$ variables $x_1,x_2,x_3$.

$n=6$: 
minumum of coefficients $=1$, maximum of coefficients = 
6217712.
Polynomial of degree 24 in 5 variables with 7993 non-zero coefficients.

$n=8$: 
minimum of coefficients $=1$, maximum of coefficients = 
1974986029814430328.
Polynomial of degree 48 in 7 variables with 
8863399
non-zero coefficients.

For $\tilde P$:

$n=4$: minimum of coefficients $=1$, maximum of coefficients=2.
Total of 5 non-zero coefficients in the $n-2$ variables $\tilde x_1, 
\tilde x_2$. 
The pfaffian is the polynomial of degree $(n-2)^2 = 4$
\[
\tilde x_1^4 + 2 \tilde x_1^3 \tilde x_2 + \tilde x_1^2 \tilde x_2^2 + 2 \tilde x_1 \tilde x_2^3 + \tilde x_2^4.
\]

$n=6$: minimum of coefficients $=1$, maximum of coefficients = 3018.
Total of 519 non-zero coefficients in the $4$ variables of degree $(n-2)^2=16$. 

$n=8$ minimum of coefficients $=1$, maximum of coefficients = 922577565632.
Total of 306016 non-zero coefficients in $n-2$.
Degree = $(n-2)^2 = 36$.

If $n=10$, then the number of perfect matchings is 
$\dfrac{10!}{2^{5} (5)!}=945$: for each one a polynomial of degree $64$ in 
8 variables is added. So, in theory computations  even in dense multivariate
polynomials with integer coefficients could fit into the memory of a normal computer. 
The minimum of the coefficients is $=1$, the maximum is 818182204944918819340996488. 
There are a total of 488783941 non-zero coefficients (the runtime was approximately 10 days). 
\end{proof}

For $n=12$, an empirical estimate of the time needed to perform the calculation
with this algorithm would be of the order of 4-5 years on the same computer. 

\section{Positive masses}

Consider now the inverse problem with real and positive masses:
let $X_0 \subset E^n=\RR^n$ be the subset 
$X_0 = \{ \vq \in E^n : \sum_{j=1}^n q_j = 0$,
 which is the orthogonal complement of $\vL$ in $E^n$.   
The $n$ columns of the anti-symmetric matrix $Q$ (which can
be denoted as $\vQ_1,\ldots, \vQ_n$) generate a subspace of dimension $n$ 
(for even $n$) or $n-1$ (for odd $n$) in $E^n$.
Let $\Pi$ denote the orthogonal projection of $E^n$ onto $X_0$: 
then if $\vx\in X_0$, equation \eqref{eq:main1d} is equivalent to 
\begin{equation}\label{eq:round1}
Q(\vx) \vm  + c \vL = \vx \iff \vx = \Pi Q(\vx) \vm.
\end{equation}
In fact, if $\vx =  Q(\vx) \vm + c \vL $, then by projecting 
one obtains $\Pi \vx = \vx = \Pi Q(\vx) \vm$  since $\Pi \vL = \zero$. 
Conversely, if $\vx = \Pi Q(\vx) \vm$,
then $\Pi Q(\vx)\vm -  Q(\vx) \in \ker \Pi =\operatorname{Span}(\vL)$, since $\Pi^2=\Pi$, and hence
there exists $c\in \RR$ such that  $\Pi Q(\vx)\vm - Q(\vx) \vm = c\vL$,
that is $\vx = Q(\vx) \vm + c\vL$.   
For a different set of variables, see 
Ouyang--Xie \cite{OuyangCollinearCentralConfiguration2005} (for $n=4$ 
bodies and $\alpha=1$) and Davis et al. \cite{DavisInverseproblemcentral2018a} (for $n=5$
bodies and $\alpha=1$); for the general problem with positive masses,
see again \cite{AlbouyInverseProblemCollinear2000}.

Now, define the following coefficients, for $i=1,\ldots, n$ and $j=0,\ldots, n-1$:
\begin{equation}
\label{eq:betas}
\beta_{ij} =\begin{cases}
1 & \text{ if } j=0~; \\
1-\frac{j}{n} & \text{ if } i\leq j~; \\
-\frac{j}{n} & \text{ if } i > j~. 
\end{cases}
\end{equation}
Consider the $n$ variables $x_0,x_1,\ldots, x_{n-1}$, where as above 
$x_j = q_j - q_{j-1}$ for $j=1,\ldots, n-1$), 
and $x_0 = \dfrac{1}{n}(q_1+\ldots+q_n)$. 
Note that for each $i=1,\ldots, n$ and $j=2,\ldots, n-1$ one has 
\[
\beta_{ij}-\beta_{i,j-1} = 
\begin{cases}
-\frac{j}{n} + \frac{j-1}{n} = -\frac{1}{n} & \text{ if } j < i \\
\frac{n-1}{n} & \text{ if } j = i \\
-\frac{1}{n}  & \text{ if } j > i \\
\end{cases}
\]
and therefore,
since $\beta_{i0}=1$,  for each $i=1\ldots n$ 
the following identities hold
\begin{equation}
\label{eq:betas2}
q_i = \sum_{j=0\cdots n-1} \beta_{ij} x_j
\quad \& \quad  x_0=\frac{1}{n}\sum_{i=1\cdots n} q_i, 
j>0\implies x_j = q_j - q_{j+1}. 
\end{equation}

Equation \eqref{eq:betas2} can be written in matrix form as follows
\begin{lemma}
\label{lemma:betas2}
Let $B$ be the matrix with coefficients $b_{ij} = \beta_{i,j-1}$ 
defined above,
$\vx$ the column vector with components $x_0,\ldots, x_{n-1}$ 
and $\vq$ the column vector with components $q_1,\ldots, q_n$. 
Then  $B$ is an invertible matrix such that $\vq = B \vx$. 
\end{lemma}

Given equation \eqref{eq:round1}, and the permutation symmetries of the potential, 
we can restrict the problem to the cone
\[
X_0^+ = \{ \vq \in X_0 : q_1>q_2> \ldots > q_n\},
\]
which in coordinates $\vx$ can be written as
\[
X_0^+ = \{ \vx : x_0 = 0, x_i>0, i=1,\ldots, n-1\}.
\]
In such coordinates, equation \eqref{eq:round1} is transformed in 
\begin{equation}
\label{eq:round2}
x_i = (B^{-1}Q\vm)_i, \quad i=1,\ldots, n-1,
\end{equation}
with suitable substitutions in the expressions of $Q$. 
For example, if $n=3$ one has to consider only the second and third rows of the 
following equation
\[
\begin{aligned}
\begin{bmatrix}
x_0\\x_1\\x_2
\end{bmatrix}
& 
=
\begin{bmatrix}
1/3 & 1/3 & 1/3 \\
1 & -1 & 0 \\
0 & 1 & -1
\end{bmatrix}
\begin{bmatrix}
0 & Q_{12} & Q_{13} \\
-Q_{12} &  0 & Q_{23} \\
-Q_{13} & -Q_{23} & 0 
\end{bmatrix}
\begin{bmatrix}
m_1\\m_2\\m_3
\end{bmatrix}
,
\end{aligned}
\]
which turns out to be
\[
\begin{bmatrix}
x_1 \\ x_2 
\end{bmatrix}
= 
\begin{bmatrix}
Q_{12} &         
Q_{12} & 
Q_{13} - Q_{23} 
\\
-Q_{12} +Q_{13} & 
Q_{23} & 
Q_{23}
\\
\end{bmatrix}
\begin{bmatrix} 
m_1\\m_2\\m_3
\end{bmatrix}~.
\]
As above, $Q_{ij} = q_{ij}^{-\alpha-1}$, for $i<j$, and hence 
the last equation can be written as 
\[
\begin{bmatrix}
x_1 \\ x_2 
\end{bmatrix}
= 
\begin{bmatrix}
x_1^{-\alpha-1} &         
x_1^{-\alpha-1} & 
(x_1+x_2)^{-\alpha-1} - x_2^{-\alpha-1}
\\
-x_1^{-\alpha-1} + (x_1+x_2)^{-\alpha-1}
& 
x_2^{-\alpha-1} & 
x_2^{-\alpha-1}
\\
\end{bmatrix}
\begin{bmatrix} 
m_1\\m_2\\m_3
\end{bmatrix}~.
\]

Another way of writing equation \eqref{eq:round2} is as follows: if now $\vx$ denotes 
the $(n-1)$-dimensional vector of positive coordinates $x_1,\ldots, x_{n-1}>0$, 
\begin{equation}
\label{eq:round3}
\vx = \sum_{k=1}^n m_k \vY_k \text{ with } m_k >0 ,
\end{equation}
where $\vY_k$ is the $(n-1)$-dimensional vector with components
\(
Y_{ik} = Q_{i,k} - Q_{i+1,k}
\)
for $i=1,\ldots, n-1$ and $k=1,\ldots, n$. 
  Given that for each $k$
\[
\sum_{i=i\ldots n-1} Y_{ik} = 
\sum_{i=1\ldots n-1}  (Q_{i,k}-Q_{i+1,k}) = 
Q_{1,k} + Q_{k,n} > 0
\]
$\vx$ and all $\vY_k$ belong to the half-space  $x_1+x_2+\ldots + x_{n-1}>0$,
and can be centrally projected  on the hyperplane 
$x_1+x_2+\ldots + x_{n-1} = 1$. 
Let $\Delta^{n-2}$ denote the standard euclidean simplex in coordinates $x_i$,
and $X_1$ the affine subspace $X_1 = \{ \vx \in X_0  : x_1+x_2\ldots + x_{n-1} = 1 \}$. 
Let $p$ denote central projection $p(\vx) = \dfrac{\vx}{\sum_{i=1\ldots n-1}x_i }  $,
partially defined $p\from X_0 \to X_1$. 
\begin{lemma}
\label{lemma:sumYikpositive}
The vector $\vx$ is a solution of \eqref{eq:round3} 
if and only if its projection $p(\vx)$ is a solution of 
\begin{equation}
\label{eq:round4}
\vx = \sum_{k=1}^n m'_k p(\vY_k) \text{ with } m'_k >0,
\end{equation}
with $\sum_{k} m_k =1$ and $\vx\in X_1$. 
\end{lemma}
\begin{proof}
As we have seen, $p$ is well defined on $\vx$ (since all $x_j$ are positive) and 
on  all $\vY_k$. 
If 
\(
\vx = \sum_{k=1}^n m_k \vY_k(\vx) 
\)
then by homogeneity if we let $\lambda_0=x_1+\ldots+x_{n-1}$ 
and $\lambda_k = 
Q_{1,k} + Q_{k,n} > 0$ for each $k$, 
\[
\begin{aligned}
\sum_{k=1}^n m'_k p( \vY_k(\lambda_0^{-1} \vx) )  & = 
\sum_{k=1}^n m'_k p( \vY_k(\vx) )  = 
\sum_{k=1}^n m'_k \lambda_k^{-1} \vY_k(\vx)   \\
\implies
\sum_{k=1}^n m'_k p( \vY_k(\lambda_0^{-1} \vx) )  & =  p(\vx) = \lambda_0^{-1}\vx \iff
m'_k \lambda_k^{-1} \lambda_0 = m_k.
\end{aligned}
\]
Now, if $\vx$ and all $\vY_k$ belong to $X_1$, 
\[
1 = \sum_{j=1}^{n-1} x_j = 
\sum_{j=1}^{n-1}\sum_{k=1}^n  m_k Y_{jk}  = 
\sum_{k=1}^n  
m_k \sum_{j=1}^{n-1}
 Y_{jk}  = 
 \sum_{k=1}^n m_k .
\]
\end{proof}

We can summarize the above facts in the following theorem. 
\begin{theo}
\label{theo:mainpositive}
Let $f\from \Delta^{n-2} \multimap X_1$ the multi-valued map defined as follows:
$f(\vx)=\CH[\vY_1(\vx),\ldots, \vY_n(\vx)] $ is 
the convex hull of the $n$ points $\vY_1,\ldots, \vY_n$ in $X_1$. 
Then $\vx \in f(\vx)$ if and only if any corresponding configuration 
$\vq$ solves the inverse central configuration problem.
\end{theo}

\begin{example}
The case $n=3$ as expected is rather simple: given that $x_1+x_2=1$, the matrix $Y$ turns
out to be 
\[
\begin{bmatrix}
x_1^{-\alpha-1} &         
x_1^{-\alpha-1} & 
1 - x_2^{-\alpha-1}
\\
1 -x_1^{-\alpha-1} 
& 
x_2^{-\alpha-1} & 
x_2^{-\alpha-1}
\\
\end{bmatrix},
\]
and the projections on  $p(\vY_k)$ on $X_1$ are the columns of the following matrix
\[
\begin{bmatrix}
x_1^{-\alpha-1} &         
\dfrac{x_1^{-\alpha-1}}{x_1^{-\alpha-1} + x_2^{-\alpha-1}}  & 
1 - x_2^{-\alpha-1}
\\
1 -x_1^{-\alpha-1} 
& 
\dfrac{x_2^{-\alpha-1}}{x_1^{-\alpha-1} + x_2^{-\alpha-1}}  & 
x_2^{-\alpha-1}
\\
\end{bmatrix}~.
\]
Given that for each $x_1\in (0,1)$ 
\[
1-x_2^{-\alpha-1} < 0<  x_1 < 1 < x_1^{-\alpha-1},
\]
for each $\vx=(x_1,x_2)\in \Delta^{1}$ 
  one has 
  $\vx \in \CH[\vY_1,\vY_3] \subset f(\vx)$,
  and hence there are positive masses solving the inverse central configuration problem. 
\end{example}

\begin{example}
Consider the case $n=4$, and $\alpha>0$. The matrix $Y$,
given that $x_1+x_2+x_3=1$, 
\[\scriptsize
\begin{aligned}
Y & = 
\begin{bmatrix}
Q_{11} - Q_{21} & Q_{12} - Q_{22} & Q_{13}-Q_{23} & Q_{14}-Q_{24} \\
Q_{21} - Q_{31} & Q_{22} - Q_{32} & Q_{23}-Q_{33} & Q_{24}-Q_{34} \\
Q_{31} - Q_{41} & Q_{32} - Q_{42} & Q_{33}-Q_{43} & Q_{34}-Q_{44} \\
\end{bmatrix}\\
&= 
\begin{bmatrix}
x_1^{-\alpha-1} & x_{1}^{-\alpha-1}  & (x_1+x_2)^{-\alpha-1} - x_2^{-\alpha-1} & 1 -(x_2+x_3)^{-\alpha-1} \\
-x_1^{-\alpha-1} + (x_1+x_2)^{-\alpha-1}  & 
x_2^{-\alpha-1}  & x_2^{-\alpha-1}  & (x_2+x_3)^{-\alpha-1} -x_3^{-\alpha-1} \\
1 - (x_1+x_2)^{-\alpha-1}   & -x_2^{-\alpha-1} + (x_2+x_3)^{-\alpha-1}  & x_3^{-\alpha-1} & x_3^{-\alpha-1} \\
\end{bmatrix}\\
\end{aligned}
\]
The projections on $X_1$ are 
\[
p(\vY_1) = \vY_1, \quad p(\vY_4) = \vY_4
\]
and
\[
p(\vY_2) = \frac{\vY_2 }{ x_1^{-\alpha-1} +(x_2+x_3)^{-\alpha-1}}, \quad
p(\vY_3) = \frac{\vY_3 }{ x_3^{-\alpha-1} +(x_1+x_2)^{-\alpha-1}}.
\]
Note that the second components of $p(\vY_1)$ and $p(\vY_4)$ are negative: 
\[
-x_1^{-\alpha-1} + (x_1+x_2)^{-\alpha-1} < 0, \quad
 (x_2+x_3)^{-\alpha-1} -x_3^{-\alpha-1} < 0 .
\]
The second components of $p(\vY_2)$ and $p(\vY_3)$ are 
\[
 \dfrac{ x_2^{-\alpha-1} } { x_1^{-\alpha-1} +(x_2+x_3)^{-\alpha-1} } \quad \text{ and } 
 \dfrac{ x_2^{-\alpha-1} } { x_3^{-\alpha-1} +(x_1+x_2)^{-\alpha-1} }.
\]
If $x_2>\frac{1}{2}$, then $x_2^{-\alpha-1} < 2^{\alpha+1}$; since $x_1+x_2+x_3=1$, 
and by convexity
\[
\begin{aligned}
 x_1^{-\alpha-1} +(x_2+x_3)^{-\alpha-1} & = x_1^{-\alpha-1} + (1-x_1)^{-\alpha-1} > 2^{\alpha+2}  \\
 x_3^{-\alpha-1} +(x_1+x_2)^{-\alpha-1} & = x_3^{-\alpha-1} + (1-x_3)^{-\alpha-1} > 2^{\alpha+2}.  \\
 \end{aligned}
\]
Hence, if $x_2>1/2$ the second components of $p(\vY_2)$ and $p(\vY_3)$
satisfy the inequalities
\[
\begin{aligned}
 \dfrac{ x_2^{-\alpha-1} } { x_1^{-\alpha-1} +(x_2+x_3)^{-\alpha-1} } & < 
 \dfrac{2^{\alpha+1}}{2^{\alpha+2}}  = 2^{-1} \\ 
 \dfrac{ x_2^{-\alpha-1} } { x_3^{-\alpha-1} +(x_1+x_2)^{-\alpha-1} } & < 2^{-1}
\end{aligned}
\]
But this means that for any $\vx$ with $x_2>1/2$, the second components of $p(\vY_k)$ is smaller
than $1/2$ for each $k$, and hence 
$\vx \not\in \CH[\vY_1,\vY_2,\vY_3,\vY_4]$: the
inverse problem does not have solutions in this region. 
For $\alpha=1$, a plot of the region where the inverse problem has solutions is 
represented in figure \ref{fig:main}. The four simplices are represented in figure \ref{fig:four}. 
The plane $x_1+x_2+x_3=1$ is projected to the $x_1x_2$-plane.
The symmetry $(x_1,x_2,x_3) \mapsto (x_3,x_2,x_1)$, which comes from the 
symmetry $(q_1,\ldots, q_n) \mapsto (-q_n,\ldots, -q_1)$ %
is projected to the affine reflection $(x_1,x_2) \mapsto (1-x_1-x_2,x_2)$. 

Note that if $x_1>\frac{1}{2}$, then ($x_2+x_3<1/2$ $\implies$ $(x_2+x_3)^{-\alpha-1} >2^{\alpha+1}$)
the following inequalities hold true:
\begin{equation}\label{eq:inequals}
\begin{aligned}
(x_1+x_2)^{-\alpha-1} & < x_1^{-\alpha}   \\
(x_2+x_3)^{-\alpha-1} & < x_3^{-\alpha}   \\
(x_2+x_3)^{-\alpha-1} & > 1 \\
x_1^{-\alpha-1} - (x_2+x_3)^{-\alpha-1} & = x_1^{-\alpha-1} -(1-x_1)^{-\alpha-1} < 0  \\
\end{aligned}
\end{equation}

Now write the projections $p(\vY_1), p(\vY_2), p(\vY_4)$ in barycentric coordinates
with respect to the affine frame $P_1'=(1,0,0)$, $P_2'=(1/2,1/2,0)$, $P_3'=(1/2,0,1/2)$ in $X_1$: 
\[
\begin{aligned}
p(\vY_1) & = \vY_1  =  
(2 x_1^{-\alpha-1} - 1) 
\begin{bmatrix}1\\0\\0 \end{bmatrix}  + 
2 ((x_1+x_2)^{-\alpha-1} - x_1^{-\alpha-1} ) 
\begin{bmatrix}0\\1\\0 \end{bmatrix}  +  \\& + 
2(1-(x_2+x_3)^{-\alpha-1}) 
\begin{bmatrix}0\\0\\1 \end{bmatrix} ;  
\\
p(\vY_4) & = \vY_4  =  
(1- 2 (x_2+x_3)^{-\alpha-1} ) 
\begin{bmatrix}1\\0\\0 \end{bmatrix}  + \\& + 
2 ((x_2+x_3)^{-\alpha-1} - x_3^{-\alpha-1} ) 
\begin{bmatrix}0\\1\\0 \end{bmatrix}  + 
2(x_3^{-\alpha-1}) 
\begin{bmatrix}0\\0\\1 \end{bmatrix}  ;
\\
\lambda 
p(\vY_2) & = \vY_2
 =  
(x_1^{-\alpha-1} - (x_2+x_3)^{-\alpha-1} ) 
\begin{bmatrix}1\\0\\0 \end{bmatrix}  + \\&+ 
2 (x_2^{-\alpha-1}) 
\begin{bmatrix}0\\1\\0 \end{bmatrix}  + 
2(-x_2^{-\alpha-1} + (x_2+x_3)^{-\alpha-1}) 
\begin{bmatrix}0\\0\\1 \end{bmatrix}  ;
\\
\end{aligned}
\]
where $\lambda= {x_2^{-\alpha-1} + (x_2+x_3)^{-\alpha-1} }  > 0$. 

Now, by inequalities  \eqref{eq:inequals}, the signs of the barycentric coordinates 
are $\vY_1 \mapsto (+,-,-)$, $(\vY_2 \mapsto (-,+,-)$, $\vY_4 \mapsto  (-,-,+)$, 
and hence the $2$-simplex $\sigma$ with vertices $P'_1$, $P'_2$ and $P'_3$ is contained in
$\CH[\vY_1,\vY_2,\vY_4]$ for each $\vx\in \sigma$, which means that 
the inverse problem has solutions.  

In fact, consider the $3\times 3$ matrix whose columns are the coordinates of 
$p\vY_1,p\vY2,p\vY_4$. It is of type
\[
A = 
\begin{bmatrix}
a_{11} & -a_{12} & -a_{13} \\
-a_{21} & a_{22} & -a_{23} \\
-a_{31} & -a_{32} & a_{33} \\
\end{bmatrix} 
\]
where the sum of the columns are $1$. 
Hence, if $D$ is the matrix with diagonal $(a_{11},a_{22},a_{33})$, 
$A=(I-\tilde A) D$, where $\tilde A$ is 
\[
\tilde A = 
\begin{bmatrix}
0 & b_{12} & b_{13} \\
b_{21} & 0  & b_{23} \\
b_{31} & b_{32} & 0  \\
\end{bmatrix} 
\]
with all $b_{ij}>$ and the sum of the columns are $<1$. 
Therefore $A^{-1} = D^{-1} \sum_{k=0}^\infty \tilde A^k$ as convergent $\norm{\tilde A}_1 <1$, 
with all entries positive. 
This implies that for each $\vx$ in the vertices of the 
triangle $x_1>1/2$ are in the interior of the $2$-simplex
$\CH[p\vY_1,p\vY_2,p\vY_4]$ (because their barycentric coordinates are proportional
to the columns of $A^{-1}$). 
\end{example}

\begin{remark}
Because of the homogeneity, one can use the following procedure to check if $\vx \in \CH[p\vY_1,\ldots, p\vY_n]$:
for each $j=1\ldots n$, compute the inverse $C_j^{-1}$  of the square matrix $C_j$ of order $n-1$ 
obtained by removing the first row and the $j$-th column of the matrix $B^{-1}Q$ (written
in terms of coordinates $x_i$). Then $\vx\in X_1$ satisfy 
$\vx \in \CH[p\vY_1,\ldots, \widehat{p\vY_k}, \ldots, p\vY_n]$ (with the $k$-th
entry removed) if and only if the vector
$C_j^{-1} \vx$ has all $n-1$ positive components,
which correspond to multiples of barycentric coordinates of $\vx$ 
with respect to the vertices 
in $\CH[p\vY_1,\ldots, \widehat{p\vY_k}, \ldots, p\vY_n]$.
\end{remark}

\begin{figure}
\centering
\includegraphics[width=0.24\textwidth]{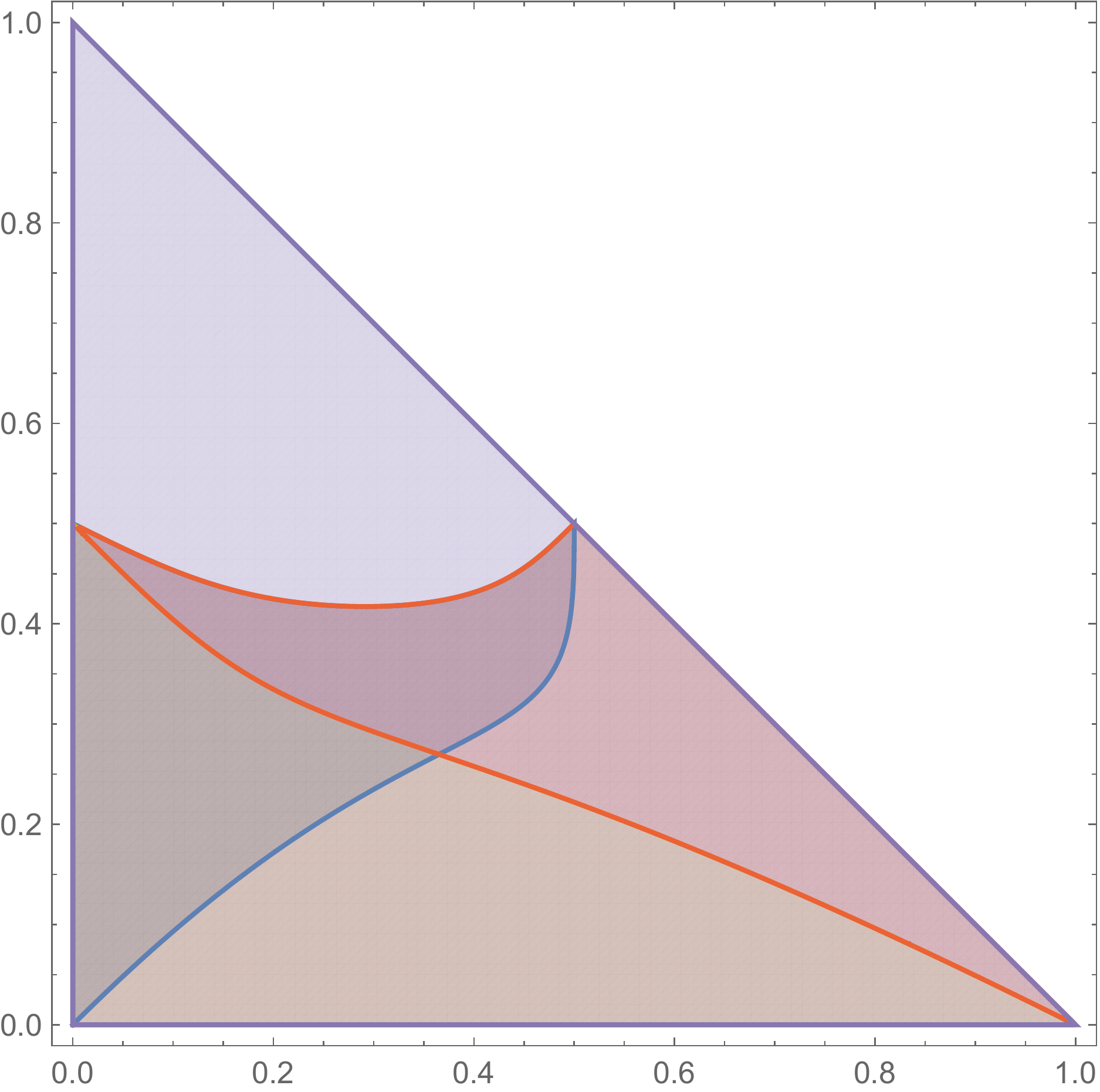}
\caption{The region of $X_1$ where $\vx \in f(\vx)$: $\vq \in [\vY_2,\vY_3,\vY_4] \cup [\vY_1,\vY_3,\vY_4]
 \cup [\vY_1,\vY_2,\vY_4] \cup  [\vY_1,\vY_2,\vY_3]$ }
\label{fig:main}
\end{figure}

\begin{figure}
    \centering
    \begin{subfigure}[t]{0.24\textwidth}
        \centering
        \includegraphics[width=\textwidth]{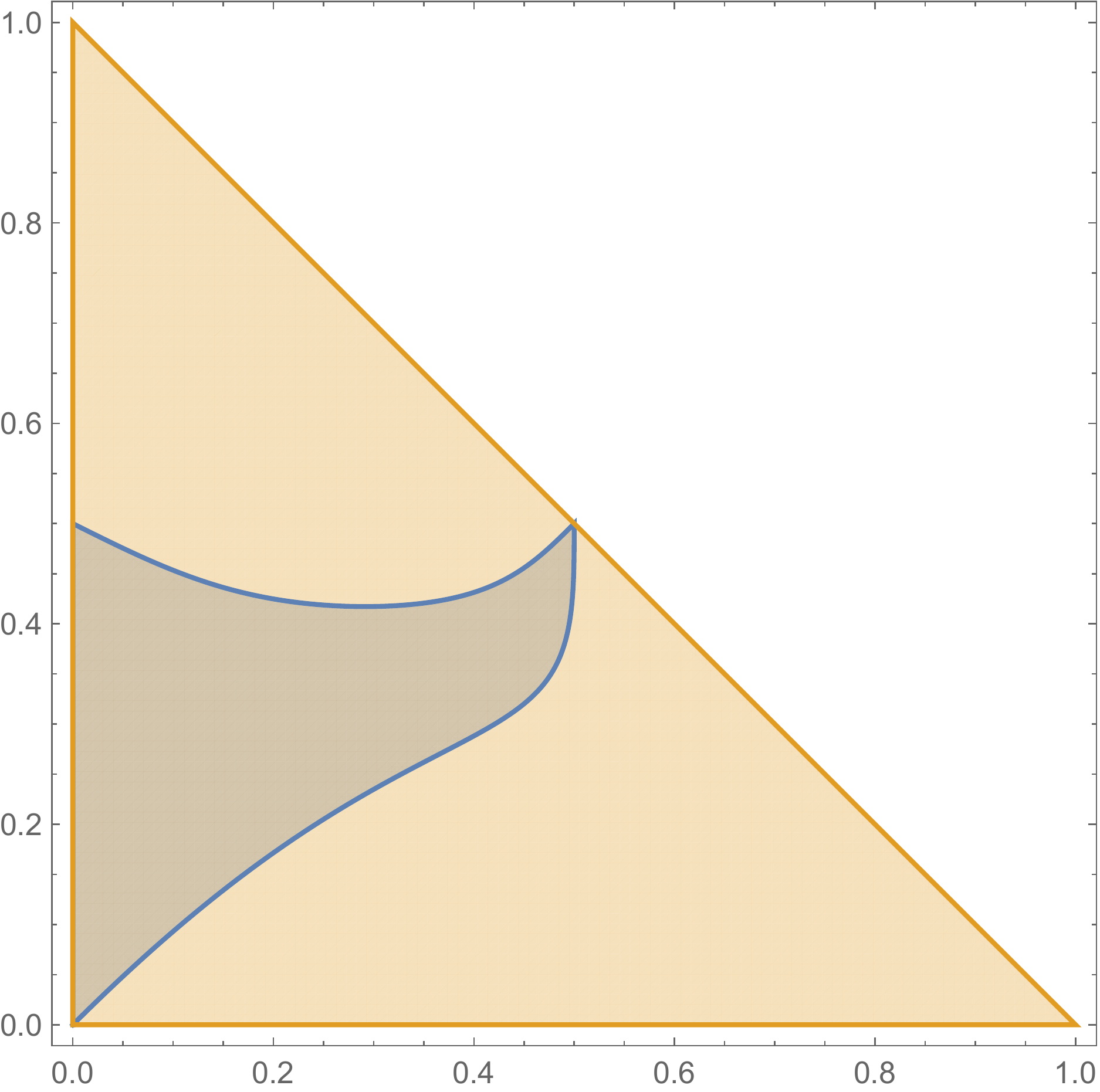}
        \caption{$\vq \in \CH[\vY_2,\vY_3,\vY_4]$}
    \end{subfigure}%
    ~ 
    \begin{subfigure}[t]{0.24\textwidth}
        \centering
        \includegraphics[width=\textwidth]{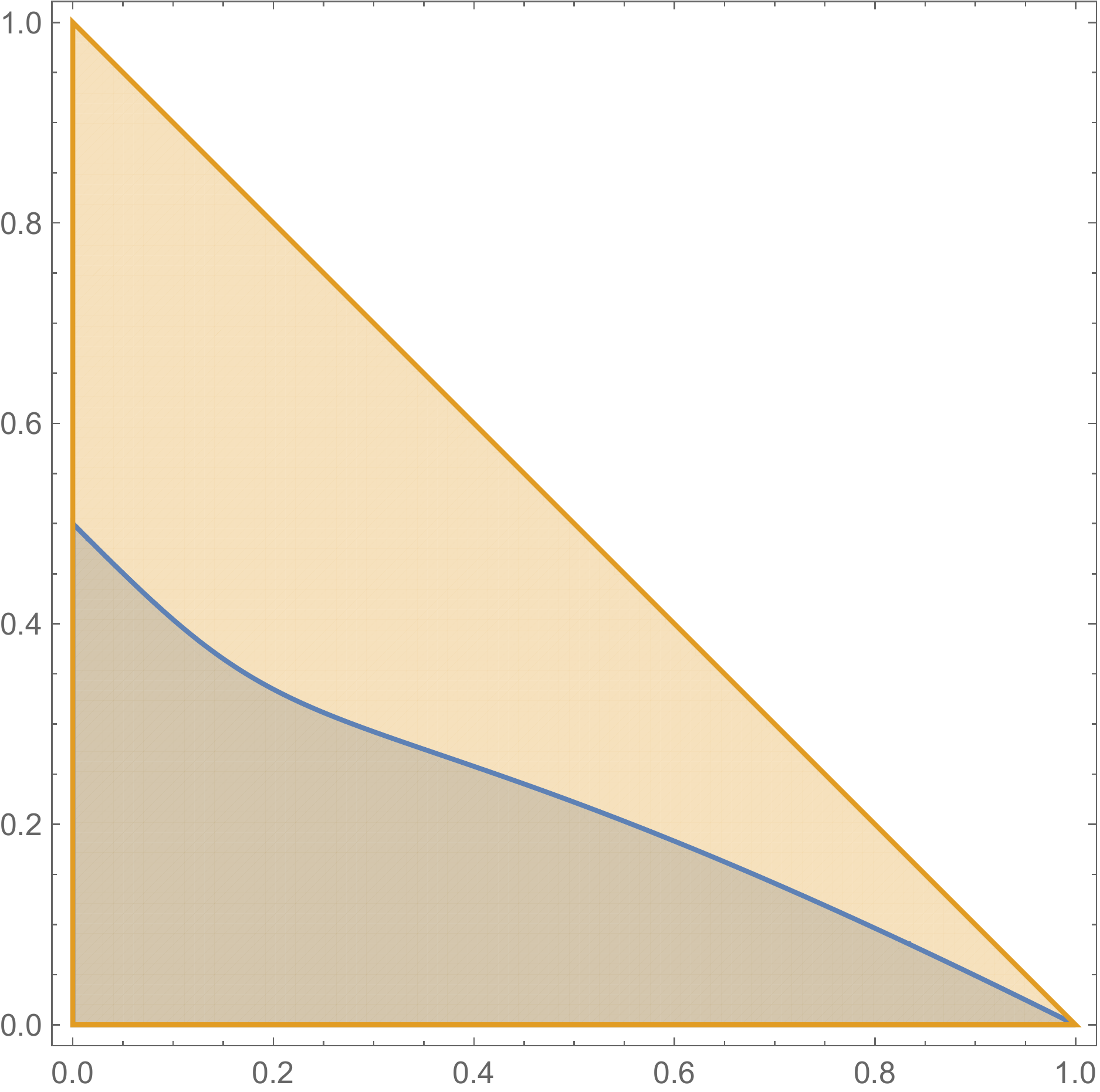}
        \caption{$\vq\in \CH[\vY_1,\vY_3,\vY_4]$}
    \end{subfigure}%
    ~ 
    \begin{subfigure}[t]{0.24\textwidth}
        \centering
        \includegraphics[width=\textwidth]{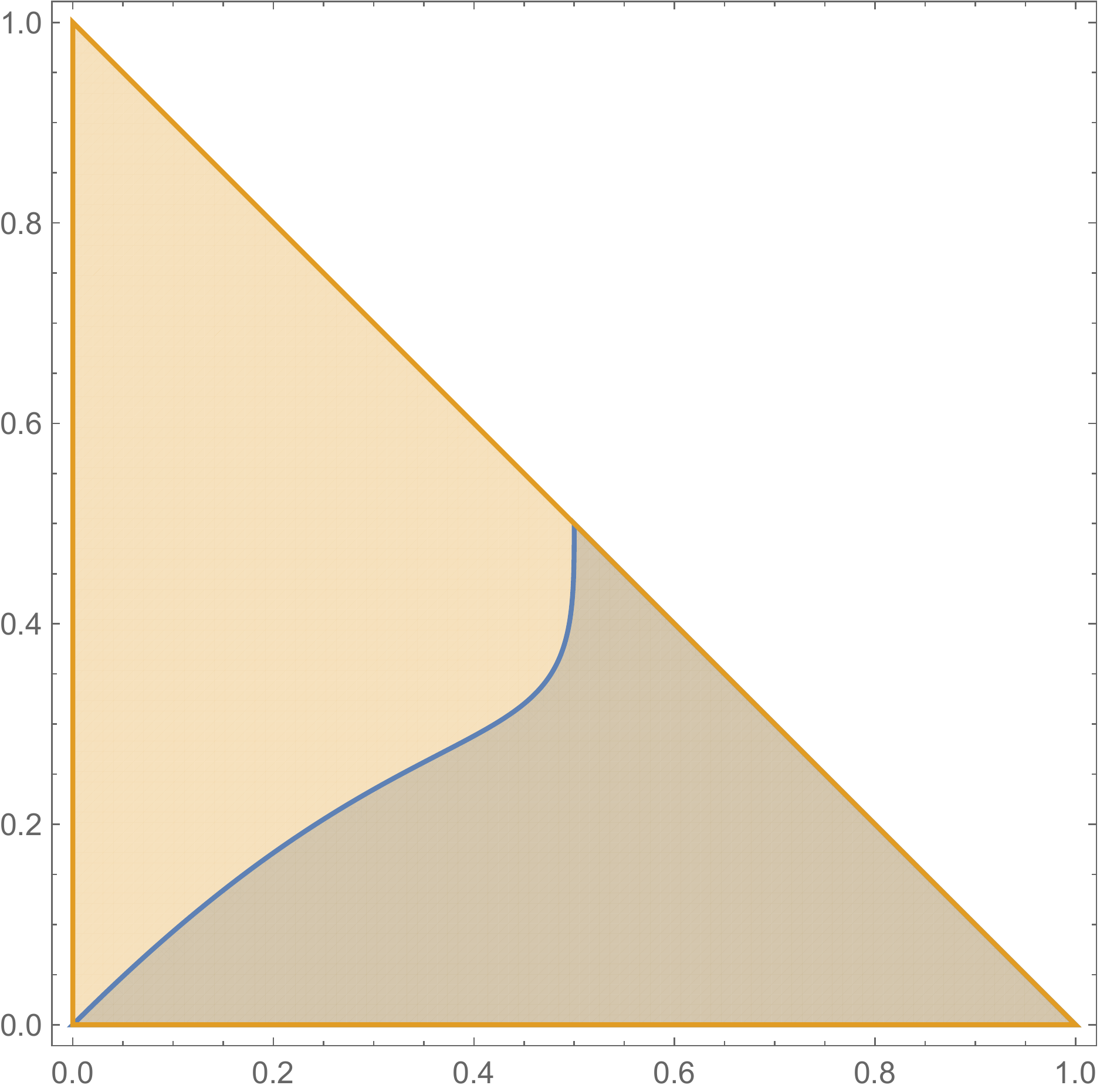}
        \caption{$\vq\in \CH[\vY_1,\vY_2,\vY_4]$}
    \end{subfigure}%
    ~ 
    \begin{subfigure}[t]{0.24\textwidth}
        \centering
        \includegraphics[width=\textwidth]{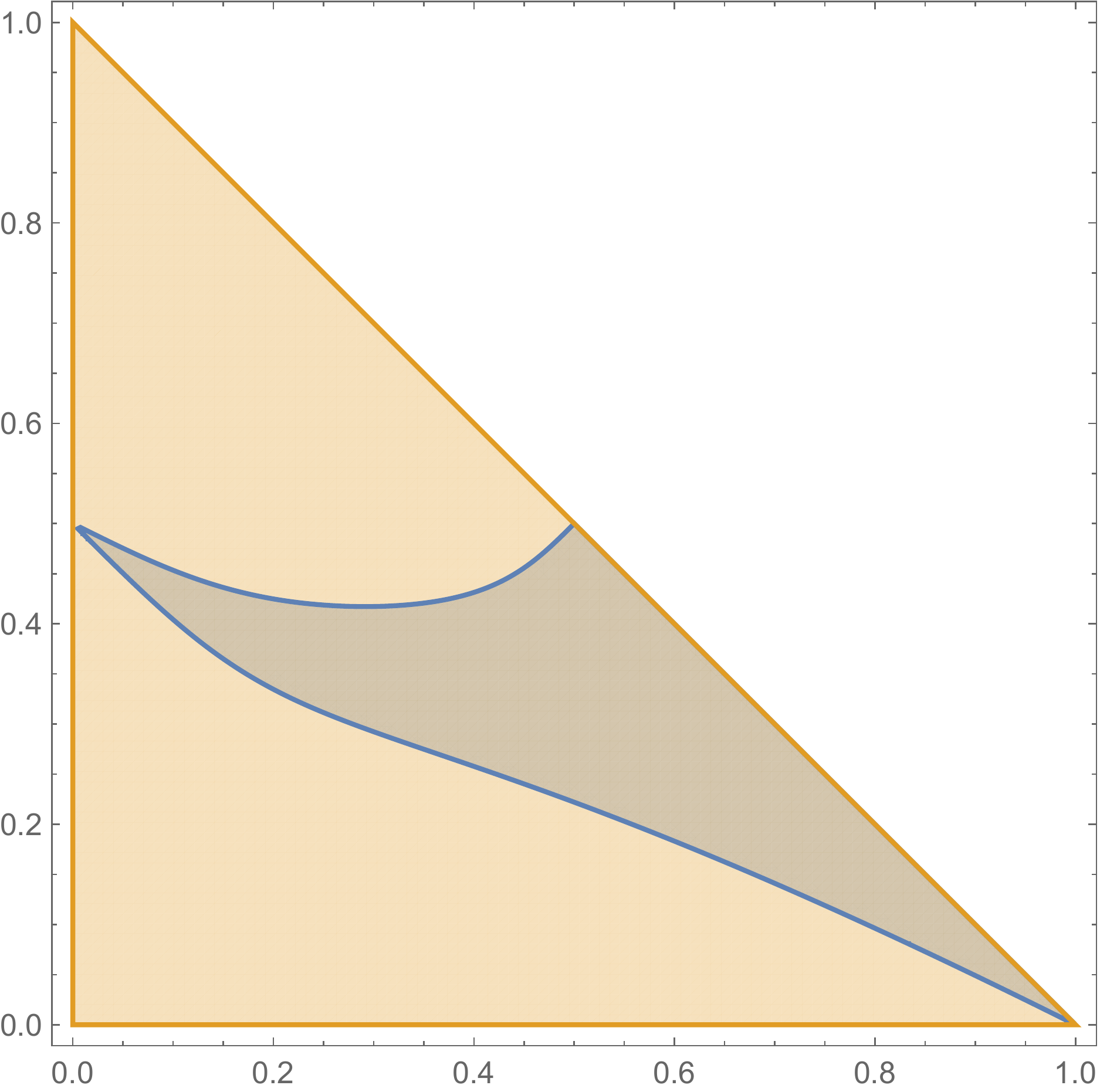}
        \caption{$\vq\in \CH[\vY_1,\vY_2,\vY_3]$}
    \end{subfigure}%
    \caption{The four regions covered by the four $2$-simplices of $\CH[\vY_1,\vY_2,\vY_3,\vY_4]$}
    \label{fig:four}
\end{figure}

\begin{theo}\label{theo:main}
Let $\vq \in \conf{n}{\RR}$ be a collinear configuration such that 
$q_1>q_2>\ldots> q_n$. 
If for an index $j$ with $2\leq j \leq n-2$  the inequality 
$2(q_j-q_{j+1}) > q_1 - q_n$ holds true, then 
the inverse problem does not 
have solutions for this configuration $\vq$: 
no positive masses $m_j$ exist such that 
$\vq$ is a central configuration with respect to the masses $m_j$.
\end{theo}
\begin{proof}
The assertion follows if we prove that if for some $i$ such that $2\leq i \leq n-2$ 
the inequality $x_i> 1/2 $ holds for the point $\vx \in X_1$ 
defined with coordinates $x_i = \dfrac{q_i-q_n}{q_1-q_n}$, then $\vx$  does not belong 
to $\CH[p\vY_1,\ldots, p\vY_n]$. 
In fact, consider the matrix $\bar Y$ with columns the vectors $p\vY_k$: 
its coefficients are, for $j=1,\ldots, n-1$ and $k=1,\ldots, n$, 
\[
\bar Y_{jk} = \dfrac{ Q_{j,k} - Q_{j+1,k}}{Q_{1k}+Q_{kn}}
\]
If $x_i > \frac{1}{2}$, for some $2\leq i\leq n-2$,
then consider the terms
$Y_{ik}$:
if $k\in \{i,i+1\}$, then 
 $Q_{1k} + Q_{kn} = (x_1+\ldots+x_{k-1})^{-\alpha-1} + (x_k + \ldots + x_n)^{-\alpha-1} > 2 ^{\alpha+1}$
 by convexity, 
 and   $Q_{i,i+1}= x_i^{-\alpha-1} < 2 ^ {\alpha+1} $ by monotonicity; 
hence the following inequalities hold
\[
Y_{ik} = \dfrac{Q_{ik} - Q_{i+1,k}}{Q_{1k}+Q_{kn}} = 
\begin{cases}
\dfrac{ -Q_{ki} + Q_{k,i+1}}{Q_{1k}+Q_{kn}} < 0< \frac{1}{2}   & \text{ if } k<  i \\
\dfrac{Q_{i,i+1}}{Q_{1i}+Q_{in}} < \frac{1}{2}  & \text{ if } k=i \\
\dfrac{Q_{i,i+1}}{Q_{1,i+1}+Q_{i+1,n}}  < \frac{1}{2}  & \text{ if } k=i+1 \\
\dfrac{Q_{ik}-Q_{i+1,k} }{Q_{1k}+Q_{kn}} < 0 < \frac{1}{2}   & \text{ if } k>i+1. \\
\end{cases}
\]
Since all the $i$-th coordinates of the points $p\vY_k$ are less than $\frac{1}{2}$,
while $x_i>\frac{1}{2}$, the point $\vx$ does not belong to 
  $\CH[p\vY_1,\ldots, p\vY_n]$. 
\end{proof}

\end{document}